\documentclass{amsart}
\usepackage{amssymb}
\usepackage{amsmath}
\usepackage{amsfonts}

\newtheorem{theorem}{Theorem}[section]

\newtheorem{lemma}[theorem]{Lemma}
\newtheorem{definition}[theorem]{Definition}
\newtheorem{proposition}[theorem]{Proposition}
\newtheorem{remark}[theorem]{Remark}

\numberwithin{theorem}{section}
\newtheorem{acknowledgement}{Acknowledgement}

\begin{document}
\title[Boundary representations of locally $C^{\ast }$-algebras]{Local
boundary representations for local operator systems }
\author{Maria Joi\c{t}a}
\address{Department of Mathematics, Faculty of Applied Sciences, University
Politehnica of Bucharest, 313 Spl. Independentei, 060042, Bucharest, Romania
and, Simion Stoilow Institute of Mathematics of the Romanian Academy, P.O.
Box 1-764, 014700, Bucharest, Romania}
\email{mjoita@fmi.unibuc.ro and maria.joita@mathem.pub.ro}
\urladdr{http://sites.google.com/a/g.unibuc.ro/maria-joita/}
\subjclass[2000]{46L07; 46L05}
\keywords{locally $C^{\ast }$-algebras, quantized domains, local completely
positive maps, local operator systems, local boundary representations}
\thanks{This work was supported by a grant of the Ministry of Research,
Innovation and Digitization, CNCS/CCCDI--UEFISCDI, project number
PN-III-P4-ID-PCE-2020-0458, within PNCDI III}

\begin{abstract}
In this paper, we show that the local boundary representations of a local
operator system in a Fr\'{e}chet locally $C^{\ast }$-algebra on quantized Fr%
\'{e}chet domains introduced by Arunkumar \cite{Ar} are in fact local
boundary representations on Hilbert spaces. Thus, the study of local boundary
representations for local operator systems in Fr\'{e}chet locally $C^{\ast }$%
-algebras on quantized Fr\'{e}chet domains is reduced to the study of
boundary representations for operator systems.
\end{abstract}

\maketitle

\section{Introduction}

An operator system is a self-adjoint linear subspace $\mathcal{S}$ of a
unital $C^{\ast }$-algebra that contains the unit. A boundary representation
for $\mathcal{S}$ is an irreducible representation of $C^{\ast }(\mathcal{S}%
) $, the $C^{\ast }$-algebra generated by $\mathcal{S}$, whose restriction
to $\mathcal{S}$ has the unique extension property. These are the objects
that generalize points of the Choquet boundary of a function system in $C(X)$%
, the $C^{\ast }$-algebra of all continuous complex values functions on a
Hausdorff compact space $X$. The notion of boundary representation for an
operator system was introduced by Arveson \cite{A} and studied extensively
by him \cite{A1,A2,A3}, Davidson and Kennedy \cite{DK}, Muhly and Solel \cite%
{MS} and others.

Effros and Webster \cite{EW} initiated a study of the locally convex version
of operator spaces called the local operator spaces. Lately, there has been
extensive research on the local operator spaces \cite{Ar,
BKL1,BKL2,BGK,D,DA,JM22}. Locally $C^{\ast }$-algebras are generalizations
of $C^{\ast }$-algebras, the topology on a locally $C^{\ast }$-algebra is
defined by a separating family of $C^{\ast }$-seminorms instead of a $%
C^{\ast }$-norm. A locally $C^{\ast }$-algebra can be identified with a
certain $\ast $-algebra of unbounded linear operators on a Hilbert space
\cite{D,I}. Dosiev \cite{D} realized local operator spaces as subspaces of
the locally $C^{\ast }$-algebra $C^{\ast }(\mathcal{D}_{\mathcal{E}})$ of
unbounded operators on a quantized domain $\{\mathcal{H};\mathcal{E};%
\mathcal{D}_{\mathcal{E}}\}$. Arunkumar \cite{Ar} introduced the notion of
local boundary representation for a local operator system and investigated
certain properties of them. A local boundary representation for a local
operator system $\mathcal{S}$ is an irreducible local representation $\pi $
of the locally $C^{\ast }$-algebra generated by $\mathcal{S}$ on a quantized
domain $\{\mathcal{H};\mathcal{E};\mathcal{D}_{\mathcal{E}}\}$ such that it
is the unique extension of the unital local $\mathcal{CP}$(completely
positive)-map $\left. \pi \right\vert _{\mathcal{S}}$ to the locally $%
C^{\ast }$-algebra generated by $\mathcal{S}$. Arunkumar \cite{Ar} limited his research on local boundary representations for local operator systems to the case of local operator systems in Fr\'{e}chet locally $C^{\ast }$-algebras because an  Arveson's extension type theorem in the context of locally $C^{\ast }$-algebras is only known for
local $\mathcal{CP}$-maps on quantized Fr\'{e}chet domains.

A local representation $\pi :\mathcal{A\rightarrow }C^{\ast }\left( \mathcal{%
D}_{\mathcal{E}}\right) $ is irreducible, in the sense of Arunkumar \cite [Definition 4.2]{Ar} if $\pi \left( \mathcal{A}\right)
^{\prime }\cap C^{\ast }\left( \mathcal{D}_{\mathcal{E}}\right) =\mathbb{C}%
I_{\mathcal{D}_{\mathcal{E}}}$, where $\pi \left( \mathcal{A}\right)
^{\prime }=\{T\in B\left( \mathcal{H}\right) ;T\pi \left( a\right) \subseteq
\pi \left( a\right) T$, for all $a\in \mathcal{A\}}$.
Unfortunately, the locally $C^{\ast }$-algebra $C^{\ast }\left( \mathcal{D}_{%
\mathcal{E}}\right) $ does not play the role of $B(\mathcal{H})$ in the
locally convex space theory. We show that the irreducible local
representations of a locally $C^{\ast }$-algebra on quantized Fr\'{e}chet
domains are in fact local representations on Hilbert spaces. Thus, the study
of local boundary representations for local operator systems is reduced to
the study of boundary representations for operator systems.

\section{Preliminaries}

\subsection{Locally $C^{\ast }$-algebras}

A \textit{locally }$C^{\ast }$\textit{-algebra }is a complete Hausdorff
complex topological $\ast $-algebra $\mathcal{A}$ whose topology is
determined by an upward filtered family $\{p_{\lambda }\}_{\lambda \in
\Lambda }\ $of $C^{\ast }$-seminorms defined on $\mathcal{A}$ (that means,
if $\lambda _{1}\leq \lambda _{2},$ then $p_{\lambda _{1}}\left( a\right)
\leq $ $p_{\lambda _{2}}\left( a\right) $ for all $a\in \mathcal{A}$). A
\textit{Fr\'{e}chet locally }$C^{\ast }$\textit{-algebra }is a locally $%
C^{\ast }$-algebra whose topology is determined by a countable family of $%
C^{\ast }$-seminorms.

A locally $C^{\ast }$-algebra $\mathcal{A}$ can be realized as a projective
limit of an inverse family of $C^{\ast }$-algebras. If $\mathcal{A}$ is a
locally $C^{\ast }$-algebra\textit{\ }with the topology determined by the
family of $C^{\ast }$-seminorms $\{p_{\lambda }\}_{\lambda \in \Lambda }$,
for each $\lambda \in \Lambda $, $\mathcal{I}_{\lambda }=\{a\in \mathcal{A}%
;p_{\lambda }\left( a\right) =0\}$ is a closed two sided $\ast $-ideal in $%
\mathcal{A}$ and $\mathcal{A}_{\lambda }=\mathcal{A}/\mathcal{I}_{\lambda }$
is a $C^{\ast }$-algebra with respect to the $C^{\ast }$-norm induced by $%
p_{\lambda }$. The canonical quotient $\ast $-morphism from $\mathcal{A}$ to
$\mathcal{A}_{\lambda }$ is denoted by $\pi _{\lambda }^{\mathcal{A}}$. For
each $\lambda _{1},\lambda _{2}\in \Lambda $ with $\lambda _{1}\leq \lambda
_{2}$, there is a canonical surjective $\ast $-morphism $\pi _{\lambda
_{2}\lambda _{1}}^{\mathcal{A}}:$ $\mathcal{A}_{\lambda _{2}}\rightarrow
\mathcal{A}_{\lambda _{1}}$, defined by $\pi _{\lambda _{2}\lambda _{1}}^{%
\mathcal{A}}\left( a+\mathcal{I}_{\lambda _{2}}\right) =a+\mathcal{I}%
_{\lambda _{1}}$ for $a\in \mathcal{A}$. Then, $\{\mathcal{A}_{\lambda },\pi
_{\lambda _{2}\lambda _{1}}^{\mathcal{A}}\}_{\lambda _{1}\leq \lambda
_{2},\lambda,\lambda _{1},\lambda _{2}\in \Lambda }$\ forms an inverse system of $%
C^{\ast }$-algebras, since $\pi _{\lambda _{1}}^{\mathcal{A}}=$ $\pi
_{\lambda _{2}\lambda _{1}}^{\mathcal{A}}\circ \pi _{\lambda _{2}}^{\mathcal{%
A}}$ whenever $\lambda _{1}\leq \lambda _{2}$. The projective limit%
\begin{equation*}
\lim\limits_{\underset{\lambda }{\leftarrow }}\mathcal{A}_{\lambda
}:=\{\left( a_{\lambda }\right) _{\lambda \in \Lambda }\in
\prod\limits_{\lambda \in \Lambda }\mathcal{A}_{\lambda };\pi _{\lambda
_{2}\lambda _{1}}^{\mathcal{A}}\left( a_{\lambda _{2}}\right) =a_{\lambda
_{1}}\text{ whenever }\lambda _{1}\leq \lambda _{2},\lambda _{1},\lambda
_{2}\in \Lambda \}
\end{equation*}%
of the inverse system of $C^{\ast }$-algebras $\{\mathcal{A}_{\lambda },\pi
_{\lambda _{2}\lambda _{1}}^{\mathcal{A}}\}_{\lambda _{1}\leq \lambda
_{2},\lambda,\lambda _{1},\lambda _{2}\in \Lambda }$ is a locally $C^{\ast }$%
-algebra that can be identified with $\mathcal{A}$ by the map $a\mapsto
\left( \pi _{\lambda }^{\mathcal{A}}\left( a\right) \right) _{\lambda \in
\Lambda }$.

\subsection{Local positive elements}

An element $a\in \mathcal{A}$ is \textit{self-adjoint} if $a^{\ast }=a$ and
it is\textit{\ positive} if $a=b^{\ast }b$ for some $b\in \mathcal{A}.$

An element $a\in \mathcal{A}$ is called \textit{local self-adjoint} if $%
a=a^{\ast }+c$, where $c\in \mathcal{A}$ such that $p_{\lambda }\left(
c\right) =0$ for some $\lambda \in \Lambda $, and we call $a$ as $\lambda $%
\textit{-self-adjoint,} and \textit{local positive} if $a=b^{\ast }b+c$
where $b,c\in $ $\mathcal{A}$ such that $p_{\lambda }\left( c\right) =0\ $
for some $\lambda \in \Lambda $, we call $a$ as $\lambda $\textit{-positive }%
and write\textit{\ }$a\geq _{\lambda }0$\textit{. }We\textit{\ }write $%
a=_{\lambda }0$ whenever $p_{\lambda }\left( a\right) =0$. Note that $a\in
\mathcal{A}$ is local self-adjoint if and only if there is $\lambda \in
\Lambda $ such that $\pi _{\lambda }^{\mathcal{A}}\left( a\right) $ is self
adjoint in $\mathcal{A}_{\lambda }$ and $a\in \mathcal{A}$ is local positive
if and only if there is $\lambda \in \Lambda $ such that $\pi _{\lambda }^{%
\mathcal{A}}\left( a\right) $ is positive in $\mathcal{A}_{\lambda }.$

\subsection{Quantized domains}

Throughout the paper, $\mathcal{H}$ is a complex Hilbert space and $B(%
\mathcal{H})$ is the $\ast $-algebra of all bounded linear operators on a
Hilbert space $\mathcal{H}$.

Let $(\Upsilon ,\leq )$ be a directed poset. A \textit{quantized domain in a
Hilbert space} $\mathcal{H}$ is a triple $\{\mathcal{H};\mathcal{E};\mathcal{%
D}_{\mathcal{E}}\}$, where $\mathcal{E}=\{\mathcal{H}_{\iota }\}_{\iota \in
\Upsilon }$ is an upward filtered family of closed nonzero subspaces such
that the union space $\mathcal{D}_{\mathcal{E}}=\bigcup\limits_{\iota \in
\Upsilon }\mathcal{H}_{\iota }$ is dense in $\mathcal{H}$ \cite{D}. If $%
\mathcal{E}$ is a countable family of closed nonzero subspaces, we say that $%
\{\mathcal{H};\mathcal{E};\mathcal{D}_{\mathcal{E}}\}$ is a \textit{%
quantized Fr\'{e}chet domain} in\textit{\ }$\mathcal{H}$.

Let $\mathcal{E}=\{\mathcal{H}_{\iota }\}_{\iota \in \Upsilon }$ be a
quantized domain in a Hilbert space $\mathcal{H}$.$\ $For each $\iota \in
\Upsilon \ $and for each $n\geq 1$, we put $\mathcal{H}^{\oplus n}:=\underset%
{n}{\underbrace{\mathcal{H}\oplus ...\oplus \mathcal{H}}},\mathcal{H}_{\iota
}^{\oplus n}:=\underset{n}{\underbrace{\mathcal{H}_{\iota }\oplus ...\oplus
\mathcal{H}_{\iota }}},$ $\mathcal{E}^{\oplus n}:=\{\mathcal{H}_{\iota
}^{\oplus n}\}_{\iota \in \Upsilon }$ and $\mathcal{D}_{\mathcal{E}^{\oplus
n}}:=\bigcup\limits_{\iota \in \Upsilon }\mathcal{H}_{\iota }^{\oplus n}$.
Then $\{\mathcal{H}^{\oplus n};\mathcal{E}^{\oplus n};\mathcal{D}_{\mathcal{E%
}^{\oplus n}}\}\ $is a quantized domain in the Hilbert space $\mathcal{H}%
^{\oplus n}\ $.

The quantized family $\mathcal{E}=\{\mathcal{H}_{\iota }\}_{\iota \in
\Upsilon }$ determines an upward filtered family $\{P_{\iota }\}_{\iota \in
\Upsilon }$ of projections in $B(\mathcal{H})$, where $P_{\iota }$ is a
projection onto $\mathcal{H}_{\iota }$.

Let
\begin{equation*}
C^{\ast }(\mathcal{D}_{\mathcal{E}}):=\{T\in \mathcal{L}(\mathcal{D}_{%
\mathcal{E}});TP_{\iota }=P_{\iota }TP_{\iota }\in B(\mathcal{H})\text{ and }%
P_{\iota }T\subseteq TP_{\iota }\text{ for all }\iota \in \Upsilon \}
\end{equation*}%
where $\mathcal{L}(\mathcal{D}_{\mathcal{E}})$ is the collection of all
linear operators on $\mathcal{D}_{\mathcal{E}}$. Let $T\in \mathcal{L}(%
\mathcal{D}_{\mathcal{E}})$. Then $T\in C^{\ast }(\mathcal{D}_{\mathcal{E}})$
if and only if $T(\mathcal{H}_{\iota })\subseteq \mathcal{H}_{\iota },T(%
\mathcal{H}_{\iota }^{\bot }\cap \mathcal{D}_{\mathcal{E}})\subseteq
\mathcal{H}_{\iota }^{\bot }\cap \mathcal{D}_{\mathcal{E}}$ and $\left.
T\right\vert _{\mathcal{H}_{\iota }}\in B(\mathcal{H}_{\iota })$ for all $%
\iota \in \Upsilon .\ $

If $T\in C^{\ast }(\mathcal{D}_{\mathcal{E}})$, then $\mathcal{D}_{\mathcal{E%
}}$ $\subseteq $ dom$(T^{\bigstar })$, where $T^{\bigstar }$ is the adjoint of $%
T,$ and $\left. T^{\bigstar }\right\vert _{\mathcal{D}_{\mathcal{E}}}\in
C^{\ast }(\mathcal{D}_{\mathcal{E}})$. Let $T^{\ast }=\left. T^{\bigstar
}\right\vert _{\mathcal{D}_{\mathcal{E}}}$. It is easy to check that $%
C^{\ast }(\mathcal{D}_{\mathcal{E}})$ equipped with this involution is a
unital $\ast $-algebra. For each $\iota \in \Upsilon $, the map $\left\Vert
\cdot \right\Vert _{\iota }:C^{\ast }(\mathcal{D}_{\mathcal{E}})\rightarrow
\lbrack 0,\infty )$,
\begin{equation*}
\left\Vert T\right\Vert _{\iota }=\left\Vert \left. T\right\vert _{\mathcal{H%
}_{\iota }}\right\Vert =\sup \{\left\Vert T\left( \xi \right) \right\Vert
;\xi \in \mathcal{H}_{\iota },\left\Vert \xi \right\Vert \leq 1\}
\end{equation*}%
is a $C^{\ast }$-seminorm on $C^{\ast }(\mathcal{D}_{\mathcal{E}})$.
Therefore, $C^{\ast }(\mathcal{D}_{\mathcal{E}})$ is a locally $C^{\ast }$%
-algebra with respect to the family of $C^{\ast }$-seminorms $\{\left\Vert
\cdot \right\Vert _{\iota }\}_{\iota \in \Upsilon }$. If $\mathcal{E}=\{%
\mathcal{H}\}$, then $C^{\ast }(\mathcal{D}_{\mathcal{E}})=B(\mathcal{H})$.

A \textit{local representation} of a locally $C^{\ast }$-algebra $\mathcal{A}
$, whose topology is defined by the family of $C^{\ast }$-seminorms $%
\{p_{\lambda }\}_{\lambda \in \Lambda },$ on a quantized domain $\{\mathcal{H%
};\mathcal{E};\mathcal{D}_{\mathcal{E}}\}$ with $\mathcal{E}=\{\mathcal{H}%
_{\iota }\}_{\iota \in \Upsilon }$ is a $\ast $-morphism $\pi :\mathcal{%
A\rightarrow }C^{\ast }(\mathcal{D}_{\mathcal{E}})$ with the property that
for each $\iota \in \Upsilon $, there exists $\lambda \in \Lambda $ such
that $\left\Vert \pi \left( a\right) \right\Vert _{\iota }\leq p_{\lambda
}\left( a\right) $ for all $a\in \mathcal{A}$. Given a locally $C^{\ast }$%
-algebra $\mathcal{A}$, whose topology is defined by the family of $C^{\ast
} $-seminorms $\{p_{\lambda }\}_{\lambda \in \Lambda }$, there exist a
quantized domain $\{\mathcal{H};\mathcal{E};\mathcal{D}_{\mathcal{E}}\}$
with $\mathcal{E}=\{\mathcal{H}_{\lambda }\}_{\lambda \in \Lambda }$ and an
local isometric representation $\pi :\mathcal{A\rightarrow }C^{\ast }(%
\mathcal{D}_{\mathcal{E}})$, that is, for each $\lambda \in \Lambda
,\left\Vert \pi \left( a\right) \right\Vert _{\lambda }=p_{\lambda }\left(
a\right) $ for all $a\in \mathcal{A}$, (see \cite[Theorem 7.2]{D} and \cite[%
Theorem 5.1]{I}). This result can be regarded as a unbounded analog of
Gelfand-Naimark theorem.

Let $\{\mathcal{H};\mathcal{E}=\{\mathcal{H}_{\iota }\}_{\iota \in \Upsilon
};\mathcal{D}_{\mathcal{E}}\}$ be a quantized domain in $\mathcal{H}$. If
the family $\{P_{\iota }\}_{\iota \in \Upsilon }$ of projections in $B(%
\mathcal{H})$ induced by $\mathcal{E}=\{\mathcal{H}_{\iota }\}_{\iota \in
\Upsilon }$ is a mutually commuting family projections in $B(\mathcal{H})$,
we say that $\{\mathcal{H};\mathcal{E}=\{\mathcal{H}_{\iota }\}_{\iota \in
\Upsilon };\mathcal{D}_{\mathcal{E}}\}$ is a\textit{\ commutative domain} in
$\mathcal{H},$ and $C^{\ast }(\mathcal{D}_{\mathcal{E}})$ is a local convex
version of $B(\mathcal{H})$ \cite{D1}. The center $\mathcal{Z}(C^{\ast }(%
\mathcal{D}_{\mathcal{E}}))$ of $C^{\ast }(\mathcal{D}_{\mathcal{E}})$ is
the local von Neumann algebra generated by the family of projections $%
\{P_{\iota }\}_{\iota \in \Upsilon }$. Given a quantized domain $\{\mathcal{H%
};\mathcal{E};\mathcal{D}_{\mathcal{E}}\}$, there exists a commutative
domain $\{\mathcal{H};\mathcal{E}_{c};\mathcal{D}_{\mathcal{E}_{c}}\}$ such
that $C^{\ast }(\mathcal{D}_{\mathcal{E}})$ identifies with a locally $%
C^{\ast }$-subalgebra in $C^{\ast }(\mathcal{D}_{\mathcal{E}_{c}})$ $\ $\cite%
[Proposition 3.1]{D1}.

\subsection{Local completely positive maps}

For each $n\geq 1,$ $M_{n}(\mathcal{A})$ denotes the collection of all
matrices of order $n$ with elements in $\mathcal{A}$. Note that $M_{n}(%
\mathcal{A})$ is a locally $C^{\ast }$-algebra, the associated family of $%
C^{\ast }$-seminorms being denoted by $\{p_{\lambda }^{n}\}_{\lambda \in
\Lambda }$, and $M_{n}(C^{\ast }(\mathcal{D}_{\mathcal{E}}))$ can be
identified with $C^{\ast }(\mathcal{D}_{\mathcal{E}^{\oplus n}}).$

For each $n\in \mathbb{N}$, the $n$-amplification of a linear map $\varphi :%
\mathcal{A}\rightarrow C^{\ast }(\mathcal{D}_{\mathcal{E}})$ is the map $%
\varphi ^{\left( n\right) }:M_{n}(\mathcal{A})$ $\rightarrow $ $C^{\ast }(%
\mathcal{D}_{\mathcal{E}^{\oplus n}})$ defined by
\begin{equation*}
\varphi ^{\left( n\right) }\left( \left[ a_{ij}\right] _{i,j=1}^{n}\right) =%
\left[ \varphi \left( a_{ij}\right) \right] _{i,j=1}^{n}
\end{equation*}%
for all $\left[ a_{ij}\right] _{i,j=1}^{n}\in M_{n}(\mathcal{A})$ .

A linear map $\varphi :\mathcal{A}\rightarrow C^{\ast }(\mathcal{D}_{%
\mathcal{E}})$ is called :

\begin{enumerate}
\item \textit{positive} if $\varphi \left( a\right) \ $is positive in $%
C^{\ast }(\mathcal{D}_{\mathcal{E}})\ $whenever $a\ $is positive in $%
\mathcal{A};$

\item \textit{completely positive }if $\varphi ^{\left( n\right) }\left( %
\left[ a_{ij}\right] _{i,j=1}^{n}\right) $ is positive in $C^{\ast }(%
\mathcal{D}_{\mathcal{E}^{\oplus n}})$ whenever $\left[ a_{ij}\right] _{i,j=1}^{n}$
is positive in $M_{n}(\mathcal{A})$ for all $n\geq 1;$

\item \textit{local positive} if for each $\iota \in \Upsilon $, there is $%
\lambda \in \Lambda $ such that $\left. \varphi \left( a\right) \right\vert
_{\mathcal{H}_{\iota }}\ $is positive in $B\left( \mathcal{H}_{\iota
}\right) \ $ whenever $ a \geq _{\lambda} 0$
and $\left. \varphi \left( a\right) \right\vert _{\mathcal{H}_{\iota }}=0\ $%
whenever $a =_ {\lambda}0;$

\item \textit{local completely positive (local }$\mathcal{CP}$\textit{) }if
for each $\iota \in \Upsilon $, there is $\lambda \in \Lambda $ such that $%
\left. \varphi ^{\left( n\right) }\left( \left[ a_{ij}\right]
_{i,j=1}^{n}\right) \right\vert _{\mathcal{H}_{\iota }^{\oplus n}}\ $is
positive in $B\left( \mathcal{H}_{\iota }^{\oplus n}\right) $ whenever $\left [\left( a_{ij}\right) \right]_{i,j=1}^{n}\geq _\lambda 0$ and $\left. \varphi ^{\left( n\right) }\left( \left[
a_{ij}\right] _{i,j=1}^{n}\right) \right\vert _{\mathcal{H}_{\iota }^{\oplus
n}}=0\ \ $whenever $\left[ \left( a_{ij}\right) %
\right] _{i,j=1}^{n}=_\lambda 0$,$\ $for all $n\geq 1$.
\end{enumerate}

\section{Irreducible local representations, local boundary representations
and pure local $\mathcal{CP}$-maps}

\subsection{Irreducible local representations}

Let $\mathcal{A}$ be a unital locally $C^{\ast }$-algebra with the topology
defined by the family of $C^{\ast }$-seminorms $\{p_{\lambda }\}_{\lambda
\in \Lambda }$ and $\{\mathcal{H},\mathcal{E}=\{\mathcal{H}_{\iota
}\}_{\iota \in \Upsilon },\mathcal{D}_{\mathcal{E}}\}$ be a quantized domain
in a Hilbert space $\mathcal{H}.$

\begin{definition}
\cite[Definition 4.2] {Ar} A local representation $\pi :\mathcal{A\rightarrow
}C^{\ast }\left( \mathcal{D}_{\mathcal{E}}\right) $ is irreducible if $\pi
\left( \mathcal{A}\right) ^{\prime }\cap C^{\ast }\left( \mathcal{D}_{%
\mathcal{E}}\right) =\mathbb{C}I_{\mathcal{D}_{\mathcal{E}}}$, where $\pi
\left( \mathcal{A}\right) ^{\prime }=\{T\in B\left( \mathcal{H}\right) ;T\pi
\left( a\right) \subseteq \pi \left( a\right) T$, for all $a\in \mathcal{A\}}
$.{}
\end{definition}

\begin{proposition}
\label{1} Let $\{\mathcal{H},\mathcal{E}=\{\mathcal{H}_{\iota }\}_{\iota \in
\Upsilon },\mathcal{D}_{\mathcal{E}}\}$ be a commutative domain and $\pi :%
\mathcal{A\rightarrow }C^{\ast }\left( \mathcal{D}_{\mathcal{E}}\right) $ be
a local representation. If $\pi $ is irreducible, then $\mathcal{H}_{\iota }=%
\mathcal{H}$ for all $\iota \in \Upsilon $ and $C^{\ast }\left( \mathcal{D}_{%
\mathcal{E}}\right) =B\left( \mathcal{H}\right) $.
\end{proposition}

\begin{proof}
Since $\{\mathcal{H},\mathcal{E}=\{\mathcal{H}_{\iota }\}_{\iota \in
\Upsilon },\mathcal{D}_{\mathcal{E}}\}$ is a commutative domain, for each $%
\iota \in \Upsilon ,$ $\left. P_{\iota }\right\vert _{\mathcal{D}_{\mathcal{E%
}}}\in C^{\ast }\left( \mathcal{D}_{\mathcal{E}}\right) $. On the other
hand, for each $\iota \in \Upsilon ,P_{\iota }\in B\left( \mathcal{H}\right)
$ and $P_{\iota }\pi \left( a\right) \subseteq \pi \left( a\right) P_{\iota
} $ for all $a\in \mathcal{A}$, and so $P_{\iota }\in \pi \left( \mathcal{A}%
\right) ^{\prime }$. Therefore, $\left. P_{\iota }\right\vert _{\mathcal{D}_{%
\mathcal{E}}}\in \pi \left( \mathcal{A}\right) ^{\prime }\cap C^{\ast
}\left( \mathcal{D}_{\mathcal{E}}\right) $ for all $\iota \in \Upsilon $,
and since $\pi $ is irreducible and $P_{\iota }\neq 0$, we have $P_{\iota
}=I_{\mathcal{D}_{\mathcal{E}}}$. Consequently, $\mathcal{H}_{\iota }=%
\mathcal{H}$ for all $\iota \in \Upsilon ,$ and then $C^{\ast }\left(
\mathcal{D}_{\mathcal{E}}\right) =B\left( \mathcal{H}\right) $.
\end{proof}

A local representation of a locally $C^{\ast }$-algebra $\mathcal{A}$ on a
Hilbert space $\mathcal{H}$ is in fact a continuous $\ast $-morphism from $%
\mathcal{A}$ to $B(\mathcal{H})$ and it is called a continuous $\ast $%
-representation of $\mathcal{A}$ on $\mathcal{H}$ \cite[\S\ 13]{Fr}.
Therefore, the irreducible local representations of a locally $C^{\ast }$%
-algebra $\mathcal{A}$ on commutative domains are local representations of $%
\mathcal{A}$ on Hilbert spaces. Moreover, we have the following
characterization of the irreducible local representations of a locally $%
C^{\ast }$-algebra $\mathcal{A}$ on a Hilbert space.

\begin{proposition}
\label{A} Let $\pi :\mathcal{A\rightarrow }B\left( \mathcal{H}\right) $ be a
local representation. Then the following statements are equivalent:

\begin{enumerate}
\item $\pi $ is irreducible;

\item There exist $\lambda _{0}\in \Lambda $ and an irreducible
representation of $\mathcal{A}_{\lambda _{0}}\ $on $\mathcal{H}$, $\pi
_{\lambda _{0}}:\mathcal{A}_{\lambda _{0}}\mathcal{\rightarrow }B\left(
\mathcal{H}\right) $, such that $\pi =\pi _{\lambda _{0}}\circ \pi _{\lambda
_{0}}^{\mathcal{A}}.\mathit{\ }$
\end{enumerate}
\end{proposition}

\begin{remark}
\label{11} If $\pi _{\lambda _{0}}:\mathcal{A}_{\lambda _{0}}\mathcal{%
\rightarrow }B\left( \mathcal{H}\right) $ is an irreducible representation,
then for every $\lambda \in \Lambda $ with $\lambda _{0}\leq \lambda ,$ $\pi
_{\lambda _{0}}\circ \pi _{\lambda \lambda _{0}}^{\mathcal{A}}$ is an
irreducible representation of $\mathcal{A}_{\lambda }.$
\end{remark}

\subsection{Pure local $\mathcal{CP}$-maps}

A \textit{local operator system} is a self-adjoint subspace of a unital
locally $C^{\ast }$-algebra which contains the unity of the algebra.

Let $\mathcal{A}$ be a unital locally $C^{\ast }$-algebra with the topology
defined by the family of $C^{\ast }$-seminorms $\{p_{\lambda }\}_{\lambda
\in \Lambda },\mathcal{S}$ be a local operator system in $\mathcal{A}$ and $%
\{\mathcal{H},\mathcal{E}=\{\mathcal{H}_{\iota }\}_{\iota \in \Upsilon },%
\mathcal{D}_{\mathcal{E}}\}$ be a quantized domain in a Hilbert space $%
\mathcal{H}$.

\begin{definition}
\cite[Definition 4.6]{Ar} A local $\mathcal{CP}$-map $\varphi :\mathcal{%
S\rightarrow }C^{\ast }\left( \mathcal{D}_{\mathcal{E}}\right) $ is pure, if
for any local $\mathcal{CP}$-map $\psi :\mathcal{S\rightarrow }C^{\ast
}\left( \mathcal{D}_{\mathcal{E}}\right) $ such that $\varphi -\psi $ is a
local $\mathcal{CP}$-map, there exists $t_{0}\in \lbrack 0,1]$ such that $%
\psi =t_{0}\varphi $.
\end{definition}

\begin{proposition}
Let $\{\mathcal{H},\mathcal{E}=\{\mathcal{H}_{\iota }\}_{\iota \in \Upsilon
},\mathcal{D}_{\mathcal{E}}\}$ be a commutative domain and $\varphi :$ $%
\mathcal{S\rightarrow }C^{\ast }\left( \mathcal{D}_{\mathcal{E}}\right) $ be
a unital local $\mathcal{CP}$-map. If $\varphi $ is pure, then $\mathcal{H}%
_{\iota }=\mathcal{H}$ for all $\iota \in \Upsilon $ and $C^{\ast }\left(
\mathcal{D}_{\mathcal{E}}\right) =B\left( \mathcal{H}\right) .$
\end{proposition}

\begin{proof}
Let $\iota _{0}\in \Upsilon \ $and $\psi :\mathcal{S\rightarrow }C^{\ast
}\left( \mathcal{D}_{\mathcal{E}}\right) $ be a linear map defined by $\psi
\left( a\right) =P_{\iota _{0}}\varphi \left( a\right) \left. P_{\iota
_{0}}\right\vert _{\mathcal{D}_{\mathcal{E}}}$. \ Since $\varphi $ is a
local $\mathcal{CP}$-map, there exists $\ \ \lambda _{0}\in \Lambda $\ \
such \ \ that $\left. \varphi ^{\left( n\right) }\left( \left[ a_{ij}\right]
\right) \right\vert _{\mathcal{H}_{\iota _{0}}^{\oplus n}}$ $\geq 0$
whenever $\left[ a_{ij}\right] \geq _{\lambda _{0}}0\ $and $\left. \varphi
^{\left( n\right) }\left( \left[ a_{ij}\right] \right) \right\vert _{%
\mathcal{H}_{\iota _{0}}^{\oplus n}}=0$ whenever $\left[ a_{ij}\right]
=_{\lambda _{0}}0$ for all $n$.

Let $\iota \in \Upsilon ,$ $\left( \xi _{k}\right) _{k=1}^{n}\in \mathcal{H}%
_{\iota }^{\oplus n}$ and $\left[ a_{ij}\right] \in M_{n}(\mathcal{A}),n\geq
1$. From $\ $%
\begin{equation*}
\left\langle \psi ^{\left( n\right) }\left( \left[ a_{ij}\right] \right)
\left( \xi _{k}\right) _{k=1}^{n},\left( \xi _{k}\right)
_{k=1}^{n}\right\rangle =\left\langle \varphi ^{\left( n\right) }\left( %
\left[ a_{ij}\right] \right) \left( P_{\iota _{0}}\xi _{k}\right)
_{k=1}^{n},\left( P_{\iota _{0}}\xi _{k}\right) _{k=1}^{n}\right\rangle
\end{equation*}%
and taking into account that $\left. \varphi ^{\left( n\right) }\left( \left[
a_{ij}\right] \right) \right\vert _{\mathcal{H}_{\iota _{0}}^{\oplus n}}$ $\
\geq 0$ whenever $\left[ a_{ij}\right] \geq _{\lambda _{0}}0\ $and $\left.
\varphi ^{\left( n\right) }\left( \left[ a_{ij}\right] \right) \right\vert _{%
\mathcal{H}_{\iota _{0}}^{\oplus n}}=0$ whenever $\left[ a_{ij}\right]
=_{\lambda _{0}}0$ for all $n,$ we deduce that $\left. \psi ^{\left(
n\right) }\left( \left[ a_{ij}\right] \right) \right\vert _{\mathcal{H}%
_{\iota }^{\oplus n}}\geq 0$ whenever $\left[ a_{ij}\right] \geq _{\lambda
_{0}}0\ $and $\left. \psi ^{\left( n\right) }\left( \left[ a_{ij}\right]
\right) \right\vert _{\mathcal{H}_{\iota }^{\oplus n}}=0$ whenever $\left[
a_{ij}\right] =_{\lambda _{0}}0$ for all $n$. Therefore, $\psi $ is a local $%
\mathcal{CP}$-map.

To show that $\varphi -\psi $ is a local $\mathcal{CP}$-map, let $\iota \in
\Upsilon ,\left( \xi _{k}\right) _{k=1}^{n}\in \mathcal{H}_{\iota }^{\oplus
n}$ and $\left[ a_{ij}\right] $ $\in M_{n}\left( \mathcal{A}\right) ,$ $%
n\geq 1$. We have
\begin{eqnarray*}
&&\left\langle \left( \varphi ^{\left( n\right) }\left( \left[ a_{ij}\right]
\right) -\psi ^{\left( n\right) }\left( \left[ a_{ij}\right] \right) \right)
\left( \xi _{k}\right) _{k=1}^{n},\left( \xi _{k}\right)
_{k=1}^{n}\right\rangle \\
&=&\left\langle \varphi ^{\left( n\right) }\left( \left[ a_{ij}\right]
\right) \left( \xi _{k}\right) _{k=1}^{n},\left( \xi _{k}\right)
_{k=1}^{n}\right\rangle -\left\langle \varphi ^{\left( n\right) }\left( %
\left[ a_{ij}\right] \right) \left( P_{\iota _{0}}\xi _{k}\right)
_{k=1}^{n},\left( P_{\iota _{0}}\xi _{k}\right) _{k=1}^{n}\right\rangle \\
&=&\left\langle \varphi ^{\left( n\right) }\left( \left[ a_{ij}\right]
\right) \left( \xi _{k}-P_{\iota _{0}}\xi _{k}\right) _{k=1}^{n},\left( \xi
_{k}-P_{\iota _{0}}\xi _{k}\right) _{k=1}^{n}\right\rangle .
\end{eqnarray*}%
Since $\Upsilon $ is a directed poset, there exists $\widetilde{\iota }\in
\Upsilon $ such that $\iota _{0}\leq \widetilde{\iota }$ and $\iota \leq $ $%
\widetilde{\iota }$, and then $\left( \xi _{k}-P_{\iota _{0}}\xi _{k}\right)
_{k=1}^{n}\in $ $\mathcal{H}_{\widetilde{\iota }}^{\oplus n}$. On the other
hand, since $\varphi $ is a local $\mathcal{CP}$-map, there exists $%
\widetilde{\lambda }\in \Lambda $ such that $\left. \varphi ^{\left(
n\right) }\left( \left[ a_{ij}\right] \right) \right\vert _{\mathcal{H}_{%
\widetilde{\iota }}^{\oplus n}}\geq 0$ whenever $\left[ a_{ij}\right] \geq _{%
\widetilde{\lambda }}0\ $and $\left. \varphi ^{\left( n\right) }\left( \left[
a_{ij}\right] \right) \right\vert _{\mathcal{H}_{\widetilde{\iota }}^{\oplus
n}}=0$ whenever $\left[ a_{ij}\right] =_{\widetilde{\lambda }}0$ for all $n$%
. Therefore,
\begin{equation*}
\left\langle \varphi ^{\left( n\right) }\left( \left[ a_{ij}\right] \right)
\left( \xi _{k}-P_{\iota _{0}}\xi _{k}\right) _{k=1}^{n},\left( \xi
_{k}-P_{\iota _{0}}\xi _{k}\right) _{k=1}^{n}\right\rangle \geq 0
\end{equation*}%
whenever $\left[ a_{ij}\right] \geq _{\widetilde{\lambda }}0$ and
\begin{equation*}
\left\langle \varphi ^{\left( n\right) }\left( \left[ a_{ij}\right] \right)
\left( \xi _{k}-P_{\iota _{0}}\xi _{k}\right) _{k=1}^{n},\left( \xi
_{k}-P_{\iota _{0}}\xi _{k}\right) _{k=1}^{n}\right\rangle =0
\end{equation*}%
whenever $\left[ a_{ij}\right] =_{\widetilde{\lambda }}0$ for all $n.$
Consequently, $\varphi -\psi $ is a local $\mathcal{CP}$-map. Since $\varphi
$ is pure, there exists $t_{0}\in \lbrack 0,1]$ such that $\psi
=t_{0}\varphi $. Then
\begin{equation*}
0\neq \left. P_{\iota _{0}}\right\vert _{\mathcal{D}_{\mathcal{E}}}=\psi
\left( 1_{\mathcal{A}}\right) =t_{0}\varphi \left( 1_{\mathcal{A}}\right)
=t_{0}I_{\mathcal{D}_{\mathcal{E}}}
\end{equation*}%
and so, $P_{\iota _{0}}=I_{\mathcal{H}}$. Therefore, $\mathcal{H}_{\iota }=%
\mathcal{H}$ for all $\iota \in \Upsilon $, and then $C^{\ast }\left(
\mathcal{D}_{\mathcal{E}}\right) =B\left( \mathcal{H}\right) $.
\end{proof}

So, the pure unital local $\mathcal{CP}$-maps of a locally $C^{\ast }$%
-algebra $\mathcal{A}$ on commutative domains are pure unital continuous $%
\mathcal{CP}$-maps with values in the $C^{\ast }$-algebra of all bounded
linear operators on a Hilbert space.

If $\mathcal{S}$ is a local operator system in $\mathcal{A}$, then, for each
$\lambda \in \Lambda ,$ $\mathcal{S}_{\lambda }=\pi _{\lambda }^{\mathcal{A}%
}\left( \mathcal{S}\right) $ is an operator system in $\mathcal{A}_{\lambda
} $. Moreover, if $\varphi :\mathcal{S\rightarrow }B\mathcal{(H)}$ is a
local $\mathcal{CP}$-map, then there exist $\lambda \in \Lambda $ and a $%
\mathcal{CP}$-map $\varphi _{\lambda }:\mathcal{S}_{\lambda }\rightarrow B%
\mathcal{(H)}$ such that $\varphi =\varphi _{\lambda }\circ \left. \pi
_{\lambda }^{\mathcal{A}}\right\vert _{\mathcal{S}}\ $(see \cite[Remark 3.1]%
{JM22}).

We have the following characterization for the pure local $\mathcal{CP}$%
-maps.

\begin{proposition}
\label{2} Let $\varphi :\mathcal{S\rightarrow }B\left( \mathcal{H}\right) $
be a local $\mathcal{CP}$-map. Then the following statements are equivalent:

\begin{enumerate}
\item $\varphi $ is pure;

\item There exist $\lambda _{0}\in \Lambda $ and a pure $\mathcal{CP}$-map, $%
\varphi _{\lambda _{0}}:\mathcal{S}_{\lambda _{0}}\mathcal{\rightarrow }%
B\left( \mathcal{H}\right) $, such that $\varphi =\varphi _{\lambda
_{0}}\circ \left. \pi _{\lambda _{0}}^{\mathcal{A}}\right\vert _{\mathcal{S}%
}.$\textit{\ }
\end{enumerate}
\end{proposition}

\begin{proof}
$\left( 1\right) \Rightarrow \left( 2\right) $ By \cite[Remark 3.1]{JM22},
there exist $\lambda _{0}\in \Lambda $ and a $\mathcal{CP}$-map $\varphi
_{\lambda _{0}}:\mathcal{S}_{\lambda _{0}}\rightarrow B\mathcal{(H)}$ such
that $\varphi =\varphi _{\lambda _{0}}\circ \left. \pi _{\lambda _{0}}^{%
\mathcal{A}}\right\vert _{\mathcal{S}}$. Let $\psi _{0}:$ $\mathcal{S}%
_{\lambda _{0}}\rightarrow B\mathcal{(H)}$ be a $\mathcal{CP}$-map such that
$\psi _{0}\leq \varphi _{\lambda _{0}}$. Then $\psi =\psi _{0}\circ \left.
\pi _{\lambda _{0}}^{\mathcal{A}}\right\vert _{\mathcal{S}}$ is a local $%
\mathcal{CP}$-map and $\psi \leq \varphi $. Since $\varphi $ is pure, there
exists $t_{0}\in \lbrack 0,1]$ such that $\psi =t_{0}\varphi $, and
consequently $\psi _{0}=t_{0}\varphi _{\lambda _{0}}.$

$\left( 2\right) \Rightarrow \left( 1\right) $ Let $\varphi _{\lambda _{0}}:%
\mathcal{S}_{\lambda _{0}}\mathcal{\rightarrow }B\left( \mathcal{H}\right) $
be a pure $\mathcal{CP}$-map such that $\varphi =\varphi _{\lambda
_{0}}\circ \left. \pi _{\lambda _{0}}^{\mathcal{A}}\right\vert _{\mathcal{S}%
}\ $and $\psi :$ $\mathcal{S}\rightarrow B\mathcal{(H)}$ be a local $%
\mathcal{CP}$-map such that $\psi \leq \varphi $. From $0\leq \psi \leq
\varphi $ and taking into account that $\varphi \left( a\right) \geq 0$
whenever $a\geq _{\lambda _{0}}0\ $and $\varphi \left( a\right) =0$ whenever
$a=_{\lambda _{0}}0$, we deduce that there exists a linear map $\psi
_{\lambda _{0}}:\left( \mathcal{S}_{\lambda _{0}}\right) _{+}\rightarrow B%
\mathcal{(H)}$ such that $\psi _{\lambda _{0}}\left( \pi _{\lambda _{0}}^{%
\mathcal{A}}\left( a\right) \right) =\psi \left( a\right) $, and which
extends by linearity to a $\mathcal{CP}$-map $\psi _{\lambda _{0}}:\mathcal{S%
}_{\lambda _{0}}\rightarrow B\mathcal{(H)}$. Moreover, $\psi =\psi _{\lambda
_{0}}\circ \left. \pi _{\lambda _{0}}^{\mathcal{A}}\right\vert _{\mathcal{S}%
} $ (see \cite[Remark 3.1]{JM22}). Then $\psi _{\lambda _{0}}\leq \varphi
_{\lambda _{0}}$, and since $\varphi _{\lambda _{0}}$ is pure, there exists $%
t_{0}\in \lbrack 0,1]$ such that $\psi _{\lambda _{0}}=t_{0}\varphi
_{\lambda _{0}}$, and consequently $\psi =t_{0}\varphi $.
\end{proof}

\begin{remark}
It is easy to check that if $\varphi _{\lambda _{0}}:\mathcal{S}_{\lambda
_{0}}\mathcal{\rightarrow }B\left( \mathcal{H}\right) $ is a pure $\mathcal{%
CP}$-map, then for every $\lambda \in \Lambda $ with $\lambda _{0}\leq
\lambda ,$ $\varphi _{\lambda _{0}}\circ \pi _{\lambda \lambda _{0}}^{%
\mathcal{A}}$ is a pure $\mathcal{CP}$-map on $\mathcal{S}_{\lambda }.$
\end{remark}

\subsection{Local boundary representations}

Let $\mathcal{A}$ be a unital locally $C^{\ast }$-algebra with the topology
defined by the family of $C^{\ast }$-seminorms $\{p_{\lambda }\}_{\lambda
\in \Lambda },\mathcal{S}$ be a local operator system in $\mathcal{A}$ and $%
\{\mathcal{H},\mathcal{E}=\{\mathcal{H}_{\iota }\}_{\iota \in \Upsilon },%
\mathcal{D}_{\mathcal{E}}\}$ be a quantized domain in a Hilbert space $%
\mathcal{H}.$\textit{\ }

\begin{definition}
\cite[Definition 5.5]{Ar} Let $\mathcal{S}$ be a local operator space in $%
\mathcal{A}\ $such that $\mathcal{S}$ generates $\mathcal{A}$. A local
representation $\pi :$ $\mathcal{A}\rightarrow $ $C^{\ast }\left( \mathcal{D}%
_{\mathcal{E}}\right) $ is a local boundary representation for $\mathcal{S}$
if:

\begin{enumerate}
\item $\pi $ is irreducible;

\item $\pi $ is the unique local completely positive extension of the unital
local $\mathcal{CP}$-map $\left. \pi \right\vert _{\mathcal{S}}$ to $%
\mathcal{A}$.
\end{enumerate}
\end{definition}

\begin{remark}
\label{C}According to Proposition \ref{1}, the local boundary
representations for a local operator system on commutative domains are local
representations on Hilbert spaces.
\end{remark}

In the following proposition we give a characterization of a local boundary representation on a Hilbert space for a local operator system $\mathcal{S}$ in terms of the boundary representations for the operator systems $\mathcal{S}_\lambda, \lambda\in \Lambda$. 
   
\begin{proposition}
\label{3} Let $\mathcal{S}$ be a local operator system in $\mathcal{A}$ such
that $\mathcal{S}$ generates $\mathcal{A}$ and $\pi :\mathcal{A\rightarrow }%
B\left( \mathcal{H}\right) $ be a local representation. Then the following
statements are equivalent:

\begin{enumerate}
\item $\pi \ $is a boundary representation for $\mathcal{S}$;

\item There exist $\lambda _{0}\in \Lambda $ and a boundary representation
for $\mathcal{S}_{\lambda _{0}}$, $\pi _{\lambda _{0}}:\mathcal{A}_{\lambda
_{0}}\mathcal{\rightarrow }B\left( \mathcal{H}\right) $, such that $\pi =\pi
_{\lambda _{0}}\circ \pi _{\lambda _{0}}^{\mathcal{A}}$, $\ $and for each $%
\lambda \in \Lambda $ with $\lambda _{0}\leq \lambda ,$ $\pi _{\lambda
_{0}}\circ \pi _{\lambda \lambda _{0}}^{\mathcal{A}}\ $is a boundary
representation for $\mathcal{S}_{\lambda }.$
\end{enumerate}
\end{proposition}

\begin{proof}
$\left( 1\right) \Rightarrow \left( 2\right) $ By Proposition \ref{A}, there
exists $\lambda _{0}\in \Lambda $ and an irreducible representation $\pi
_{\lambda _{0}}:\mathcal{A}_{\lambda _{0}}\mathcal{\rightarrow }B\left(
\mathcal{H}\right) $ such that $\pi =\pi _{\lambda _{0}}\circ \pi _{\lambda
_{0}}^{\mathcal{A}}$.\textit{\ }Let $\lambda \in \Lambda $ with $\lambda
_{0}\leq \lambda $. By Remark \ref{11}, $\pi _{\lambda _{0}}\circ \pi
_{\lambda \lambda _{0}}^{\mathcal{A}}$ is an irreducible representation of $%
\mathcal{A}_{\lambda }$. Let $\varphi _{\lambda }:\mathcal{A}_{\lambda }%
\mathcal{\rightarrow }B\left( \mathcal{H}\right) $ be a $\mathcal{CP}$%
\textbf{-}map such that $\left. \varphi _{\lambda }\right\vert _{\mathcal{S}%
_{\lambda }}=\left. \left( \pi _{\lambda _{0}}\circ \pi _{\lambda \lambda
_{0}}^{\mathcal{A}}\right) \right\vert _{\mathcal{S}_{\lambda }}$. Then $%
\varphi _{\lambda }\circ \pi _{\lambda }^{\mathcal{A}}$ is a local $\mathcal{%
CP}$\textbf{-}map, and
\begin{eqnarray*}
\left. \left( \varphi _{\lambda }\circ \pi _{\lambda }^{\mathcal{A}}\right)
\right\vert _{\mathcal{S}} &=&\left. \left( \pi _{\lambda _{0}}\circ \pi
_{\lambda \lambda _{0}}^{\mathcal{A}}\right) \right\vert _{\mathcal{S}%
_{\lambda }}\circ \left. \pi _{\lambda }^{\mathcal{A}}\right\vert _{\mathcal{%
S}}=\left. \pi _{\lambda _{0}}\right\vert _{\mathcal{S}_{\lambda _{0}}}\circ
\left. \pi _{\lambda _{0}}^{\mathcal{A}}\right\vert _{\mathcal{S}} \\
&=&\left. \left( \pi _{\lambda _{0}}\circ \pi _{\lambda _{0}}^{\mathcal{A}%
}\right) \right\vert _{\mathcal{S}}=\left. \pi \right\vert _{\mathcal{S}}
\end{eqnarray*}%
whence, since $\pi \ $is a local boundary representation for $\mathcal{S}$,$%
\ $we deduce that $\varphi _{\lambda }\circ \pi _{\lambda }^{\mathcal{A}%
}=\pi =\pi _{\lambda _{0}}\circ \pi _{\lambda \lambda _{0}}^{\mathcal{A}%
}\circ \pi _{\lambda }^{\mathcal{A}}$. Consequently, $\varphi _{\lambda
}=\pi _{\lambda _{0}}\circ \pi _{\lambda \lambda _{0}}^{\mathcal{A}}$.
Therefore, $\pi _{\lambda _{0}}\circ \pi _{\lambda \lambda _{0}}^{\mathcal{A}%
}$ is a boundary representation for $\mathcal{S}_{\lambda }$.

$\left( 2\right) \Rightarrow \left( 1\right) $ By Proposition \ref{A}, the
map $\pi :=\pi _{\lambda _{0}}\circ \pi _{\lambda _{0}}^{\mathcal{A}}%
\mathcal{\ }$is an irreducible local representation of $\mathcal{A}$ on $%
\mathcal{H}.$ Let $\varphi :\mathcal{A\rightarrow }B\left( \mathcal{H}%
\right) $ be a local $\mathcal{CP}$\textbf{-}map such that $\left. \varphi
\right\vert _{\mathcal{S}}=\left. \pi \right\vert _{\mathcal{S}}$. Since $%
\varphi $ is a local $\mathcal{CP}$\textbf{-}map and $\Lambda $ is a
directed poset, there exist $\lambda _{1}\in \Lambda ,\lambda _{0}\leq
\lambda _{1}$ and a $\mathcal{CP}$\textbf{-}map $\varphi _{\lambda _{1}}:%
\mathcal{A}_{\lambda _{1}}\rightarrow B\left( \mathcal{H}\right) $ such that
$\varphi =\varphi _{\lambda _{1}}\circ \pi _{\lambda _{1}}^{\mathcal{A}}$.
Then
\begin{eqnarray*}
\left. \varphi _{\lambda _{1}}\right\vert _{\mathcal{S}_{\lambda _{1}}}\circ
\left. \pi _{\lambda _{1}}^{\mathcal{A}}\right\vert _{\mathcal{S}} &=&\left.
\varphi _{\lambda _{1}}\circ \pi _{\lambda _{1}}^{\mathcal{A}}\right\vert _{%
\mathcal{S}}=\left. \varphi \right\vert _{\mathcal{S}}=\left. \pi
\right\vert _{\mathcal{S}} \\
&=&\left. \pi _{\lambda _{0}}\circ \pi _{\lambda _{1}\lambda _{0}}^{\mathcal{%
A}}\circ \pi _{\lambda _{1}}^{\mathcal{A}}\right\vert _{\mathcal{S}}=\left.
\pi _{\lambda _{0}}\circ \pi _{\lambda _{1}\lambda _{0}}^{\mathcal{A}%
}\right\vert _{\mathcal{S}_{\lambda _{1}}}\circ \left. \pi _{\lambda _{1}}^{%
\mathcal{A}}\right\vert _{\mathcal{S}}.
\end{eqnarray*}%
From the above relation and taking into account that $\pi _{\lambda
_{0}}\circ \pi _{\lambda _{1}\lambda _{0}}^{\mathcal{A}}$ is a boundary
representation for $\mathcal{S}_{\lambda _{1}}$, we conclude that $\varphi
_{\lambda _{1}}=\pi _{\lambda _{0}}\circ \pi _{\lambda _{1}\lambda _{0}}^{%
\mathcal{A}}$. Then $\varphi =\varphi _{\lambda _{1}}\circ \pi _{\lambda
_{1}}^{\mathcal{A}}=\pi _{\lambda _{0}}\circ \pi _{\lambda _{1}\lambda
_{0}}^{\mathcal{A}}\circ \pi _{\lambda _{1}}^{\mathcal{A}}=\pi $. Therefore,
$\pi \ $is a local boundary representation for $\mathcal{S}$.
\end{proof}

\begin{remark}
\textit{\ Since a} \textit{quantized Fr\'{e}chet domain }is a commutative
domain, the local boundary representations for a local \textit{operator
system } $\mathcal{S}$ \textit{in a Fr\'{e}chet locally }$C^{\ast }$-algebra considered by Arunkumar \cite{Ar} are local boundary representations
on Hilbert spaces (Proposition \ref{1}), and according to the above proposition, the proofs of the main theorems \cite[Theorems 5.7; 5.10 and 5.11]{Ar} are reduced to the case of boundary representations for operator systems   (see the following section).
\end{remark}

\begin{proposition}
Let $\mathcal{S}$ be a local operator system in $\mathcal{A}$ such that $%
\mathcal{S}$ generates $\mathcal{A}$ and $\pi :\mathcal{A\rightarrow }%
C^{\ast }\left( \mathcal{D}_{\mathcal{E}}\right) $ be a local boundary
representation for $\mathcal{S}.$ If $\{\mathcal{H},\mathcal{E}=\{\mathcal{H}%
_{\iota }\}_{\iota \in \Upsilon },\mathcal{D}_{\mathcal{E}}\}$ is a
commutative domain, then $\left. \pi \right\vert _{\mathcal{S}}:\mathcal{%
S\rightarrow }C^{\ast }\left( \mathcal{D}_{\mathcal{E}}\right) $ is a pure
local $\mathcal{CP}$-map.
\end{proposition}

\begin{proof}
By Remark \ref{C}, $C^{\ast }\left( \mathcal{D}_{\mathcal{E}}\right)
=B\left( \mathcal{H}\right) .$ Since $\pi $ is a local boundary
representation for $\mathcal{S}$, by Proposition \ref{3}, there exist $%
\lambda _{0}\in \Lambda $ and a boundary representation $\pi _{\lambda _{0}}:%
\mathcal{A}_{\lambda _{0}}\mathcal{\rightarrow }B\left( \mathcal{H}\right) $
for $\mathcal{S}_{\lambda _{0}}$, such that $\pi =\pi _{\lambda _{0}}\circ
\pi _{\lambda _{0}}^{\mathcal{A}}$. By \cite[Lemma 2.4.3]{A}, since $\pi
_{\lambda _{0}}$ is a boundary representation for $\mathcal{S}_{\lambda
_{0}},\left. \pi _{\lambda _{0}}\right\vert _{\mathcal{S}_{\lambda _{0}}}:%
\mathcal{S}_{\lambda _{0}}\mathcal{\rightarrow }B\left( \mathcal{H}\right) \
$is pure, whence, according Proposition \ref{2}, it follows that $\left. \pi
_{\lambda _{0}}\right\vert _{\mathcal{S}_{\lambda _{0}}}\circ \left. \pi
_{\lambda _{0}}^{\mathcal{A}}\right\vert _{\mathcal{S}}\ $is a pure local $%
\mathcal{CP}$-map. But $\left. \pi \right\vert _{\mathcal{S}}=\left. \pi
_{\lambda _{0}}\right\vert _{\mathcal{S}_{\lambda _{0}}}\circ \left. \pi
_{\lambda _{0}}^{\mathcal{A}}\right\vert _{\mathcal{S}}$, and the
proposition is proved.
\end{proof}

\begin{remark}
Since the quantized Fr\'{e}chet domains are commutative domains, Theorem
5.10 \cite{Ar} is a particular case of the above proposition. \textit{\ }
\end{remark}

\section{Local boundary representations on Hilbert spaces}

Let $\mathcal{A}$ be a unital locally $C^{\ast }$-algebra with the topology
defined by the family of $C^{\ast }$-seminorms $\{p_{\lambda }\}_{\lambda
\in \Lambda }$ and $\mathcal{S}$ be a local operator systems in $\mathcal{A}$. If $\mathcal{B}$ is the locally $C^{\ast }$-subalgebra of $\mathcal{A}$
generated by $\mathcal{S}$, then, clearly, for each $\lambda \in \Lambda $, $%
\mathcal{B}_{\lambda }$ is $C^{\ast }$-algebra generated by $\mathcal{S}%
_{\lambda }.$

By \cite[Theorem 3.3]{JM22}, a local $\mathcal{CP}$-map $\varphi :\mathcal{%
S\rightarrow }B\left( \mathcal{H}\right) $ extends to a $\mathcal{CP}$-map $%
\widetilde{\varphi }:\mathcal{A\rightarrow }B\left( \mathcal{H}\right) .$

Arveson \cite{A} calls the set of all boundary representations of an
operator system the non-commutative Choquet boundary. The collection of all
local boundary representations on Hilbert spaces of a local operator system
$\mathcal{S}$ is called the \textit{local non-commutative Choquet boundary for $\mathcal{S}$.}

\begin{remark}
It is known that the Choquet boundary for the unital commutative $C^{\ast }$%
-algebra $C(X)$ of all continuous complex values functions on a Hausdorff
compact space $X$ is $X.$

Let $\{X_{n};i_{nm}:X_{n}\hookrightarrow X_{m};n\leq m;n,m\in \mathbb{N}\}$
be an inductive system of Hausdorff compact spaces and $X:=\lim\limits_{%
\underset{n}{\rightarrow }}X_{n}.$ Then $C(X):=\{f:X\rightarrow \mathbb{C};f$
is continuous$\}$ is a unital commutative \textit{Fr\'{e}chet locally }$%
C^{\ast }$-algebra with respect to the topology defined by the family of $%
C^{\ast }$-seminorms $\{p_{n}\}_{n\in \mathbb{N}}$, where $p_{n}\left(
f\right) =\sup \{\left\vert f\left( x\right) \right\vert ;x\in X_{n}\}$.
Moreover, $C(X)$ can be identified with $\lim\limits_{\underset{n}{%
\leftarrow }}C\left( X_{n}\right) $ \cite{Ph}. According to Proposition \ref%
{3}, as in the case of unital commutative $C^{\ast }$-algebras, we obtained
that the local Choquet boundary for $C(X)$ is $X.$
\end{remark}

The following theorem is a local version of \cite[Theorem 2.1.2]{A}.

\begin{theorem}
Let $\mathcal{A}_{1}$ and $\mathcal{A}_{2}$ be two locally $C^{\ast }$%
-algebras with the topology defined by the family of $C^{\ast }$-seminorms $%
\{p_{\lambda }\}_{\lambda \in \Lambda },$ respectively $\{q_{\lambda
}\}_{\lambda \in \Lambda },\mathcal{S}_{1}$ and $\mathcal{S}_{2}$ be two
local operator systems such that $\mathcal{S}_{1}$ generated $\mathcal{A}%
_{1} $ and $\mathcal{S}_{2}$ generated $\mathcal{A}_{2}$. If $\Phi :\mathcal{%
S}_{1}\rightarrow \mathcal{S}_{2}$ is a unital surjective local completely
isometric linear map, then for each local boundary representation $\pi _{1}:$
$\mathcal{A}_{1}\rightarrow B(\mathcal{H})$ for $\mathcal{S}_{1},$ there
exists a local boundary representation $\pi _{2}:$ $\mathcal{A}%
_{2}\rightarrow B(\mathcal{H})$ for $\mathcal{S}_{2}$ such that $\pi
_{2}\circ \Phi =\left. \pi _{1}\right\vert _{\mathcal{S}_{1}}.$
\end{theorem}

\begin{proof}
Since $\Phi :\mathcal{S}_{1}\rightarrow \mathcal{S}_{2}$ is a unital
surjective local completely isometric linear map, for each $\lambda \in
\Lambda ,$ there exists a unital surjective local completely isometric
linear map $\Phi _{\lambda }:\left( \mathcal{S}_{1}\right) _{\lambda
}\rightarrow \left( \mathcal{S}_{2}\right) _{\lambda }$ such that $\pi
_{\lambda }^{\mathcal{A}_{2}}\circ \Phi =\Phi _{\lambda }\circ \left. \pi
_{\lambda }^{\mathcal{A}_{1}}\right\vert _{\mathcal{S}_{1}}.$

Let $\pi _{1}:$ $\mathcal{A}_{1}\rightarrow B(\mathcal{H})$ be a local
boundary representation for $\mathcal{S}_{1}$. By Proposition \ref{3}, there
exist $\lambda _{0}\in \Lambda $ and a boundary representation $\left( \pi
_{1}\right) _{\lambda _{0}}:\left( \mathcal{A}_{1}\right) _{\lambda _{0}}%
\mathcal{\rightarrow }B\left( \mathcal{H}\right) $ for $\left( \mathcal{S}%
_{1}\right) _{\lambda _{0}}$ such that $\pi _{1}=\left( \pi _{1}\right)
_{\lambda _{0}}\circ \pi _{\lambda _{0}}^{\mathcal{A}_{1}}$ and for each $\lambda \in \Lambda$ with $\lambda_0\ \leq \lambda$, $\left(\pi_1\right)_{\lambda_0}\circ \pi_{\lambda \lambda_0}^{\mathcal{A}_1}$ is a boundary representation for $\left(\mathcal{S}_1\right)_\lambda$. Then, by \cite[%
Theorem 2.1.2]{A}, for each $\lambda \in \Lambda$ with $\lambda_0 \leq \lambda $, there exists a boundary representation $\left( \pi_{2}\right) _{\lambda }:\left( \mathcal{A}_{2}\right) _{\lambda}
\rightarrow B\left( \mathcal{H}\right) $ for $\left(\mathcal{S}_{2}\right)
_{\lambda}$
 such that $\left(\pi_{2}\right)_{\lambda}\circ \Phi_{\lambda}=\left. \left(\pi _{1}\right)_{\lambda _{0}}\circ \pi_{\lambda\lambda_0}^{\mathcal{A}_1}\right\vert_{\left( \mathcal{S}_{1}\right) _{\lambda}}$.
  For each $\lambda \in \Lambda $ with $\lambda_0 \leq \lambda $, $\left(\pi_2\right)_{\lambda_0}\circ \pi_{\lambda \lambda_0}^{\mathcal{A}_2}$ is an irreducible representation of $\left(\mathcal{A}_2\right)_\lambda$ on $\mathcal{H}$. Moreover, 

\begin{equation*}
\left(\pi_2\right)_{\lambda_0}\circ \pi_{\lambda \lambda_0}^{\mathcal{A}_2}\circ \Phi_\lambda
=\left(\pi_2\right)_{\lambda_0}\circ \Phi_{\lambda_0}\circ \left. \pi_{\lambda \lambda_0}^{\mathcal{A}_1}\right\vert_{{\left(\mathcal{S}_1\right)}_{\lambda}}=\left(\pi_1\right)_{\lambda_0}\circ  \left. \pi_{\lambda \lambda_0}^{\mathcal{A}_1}\right\vert_{\left(\mathcal{S}_1\right)_\lambda}=\left(\pi_2\right)_{\lambda}\circ \Phi_\lambda 
\end{equation*}  whence we deduce that $\left.\left(\pi_2 \right)_{\lambda_0}\circ \pi_{\lambda \lambda_0}^{\mathcal{A}_2}\right\vert_{\left(\mathcal{S}_2\right)_\lambda}=\left.\left(\pi_2 \right)_{\lambda} \right\vert_{\left(\mathcal{S}_2\right)_\lambda}$, and since  $\left(\pi_2\right)_\lambda$ is a boundary representation for $ \left(\mathcal{S}_2\right)_\lambda$, it follows that $\left(\pi_2 \right)_{\lambda_0}\circ \pi_{\lambda \lambda_0}^{\mathcal{A}_2}=\left(\pi_2 \right)_\lambda$. Therefore, $\left(\pi_2 \right)_{\lambda_0}\circ \pi_{\lambda \lambda_0}^{\mathcal{A}_2}$ is a boundary representation for  $ \left(\mathcal{S}_2\right)_\lambda$. Then, $\pi_2:=\left(\pi_2\right)_{\lambda_0} \circ \pi_{\lambda_0}^{\mathcal{A}_2}$ is a local boundary representation for $\mathcal{S}_2$, and moreover,  

\begin{equation*}
\pi _{2}\circ \Phi =\left( \pi _{2}\right) _{\lambda _{0}}\circ \pi
_{\lambda _{0}}^{\mathcal{A}_{2}}\circ \Phi =\left( \pi _{2}\right)
_{\lambda _{0}}\circ \Phi _{\lambda_0 }\circ \left. \pi _{\lambda_0 }^{\mathcal{A%
}_{1}}\right\vert _{\mathcal{S}_{1}}=\left. \left( \pi _{1}\right) _{\lambda
_{0}}\right\vert _{\left( \mathcal{S}_{1}\right) _{\lambda _{0}}}\circ
\left. \pi _{\lambda_0 }^{\mathcal{A}_{1}}\right\vert _{\mathcal{S}%
_{1}}=\left. \pi _{1}\right\vert _{\mathcal{S}_{1}}.
\end{equation*}
\end{proof}

\begin{remark}
Since the local boundary representations for a local \textit{\ operator
system }$\mathcal{S}$ \textit{in a Fr\'{e}chet locally }$C^{\ast }$-algebra
on \textit{quantized Fr\'{e}chet domains }are local boundary representations
on Hilbert spaces (Proposition \ref{1}), Theorem 5.7 \cite{Ar} is a
particular case of the above theorem.
\end{remark}

A local representation $\pi :\mathcal{A\rightarrow }B\left( \mathcal{H}%
\right) $ is a \textit{finite representation} for $\mathcal{S}$ if for every
isometry $V\in B\left( \mathcal{H}\right) $, the condition $\pi \left(
a\right) =V^{\ast }$ $\pi \left( a\right) V$ for all $a\in \mathcal{S}$
implies $V$ is unitary. It is easy to check that a local representation $\pi
:\mathcal{A\rightarrow }B\left( \mathcal{H}\right) $ is a finite
representation for $\mathcal{S}$ if and only if there exist $\lambda _{0}\in
\Lambda $ and a finite representation for $\mathcal{S}_{\lambda _{0}}$, $\pi
_{\lambda _{0}}:\mathcal{A}_{\lambda _{0}}\mathcal{\rightarrow }B\left(
\mathcal{H}\right) $, such that $\pi =\pi _{\lambda _{0}}\circ \pi _{\lambda
_{0}}^{\mathcal{A}}$. Clearly, if $\pi _{\lambda _{0}}:\mathcal{A}_{\lambda
_{0}}\mathcal{\rightarrow }B\left( \mathcal{H}\right) $ is a finite
representation for $\mathcal{S}_{\lambda _{0}}$, then $\pi _{\lambda
_{0}}\circ \pi _{\lambda \lambda _{0}}^{\mathcal{A}}:\mathcal{A}_{\lambda }%
\mathcal{\rightarrow }B\left( \mathcal{H}\right) $ is a finite
representation for $\mathcal{S}_{\lambda }$ for all $\lambda \in \Lambda $
with $\lambda _{0}\leq \lambda .$

We say that $\mathcal{S}$ \textit{separates} the irreducible local
representation $\pi :\mathcal{A\rightarrow }B\left( \mathcal{H}\right) $ if
for any irreducible local representation $\rho :\mathcal{A\rightarrow }%
B\left( \mathcal{K}\right) $ and an isometry $V\in B\left( \mathcal{H},%
\mathcal{K}\right) $ such that $\pi \left( a\right) =V^{\ast }\rho \left(
a\right) V$ for all $a\in \mathcal{S}$ implies that $\pi $ and $\rho $ are
unitarily equivalent.

\begin{lemma}
\label{4}Let $\mathcal{S}$ be a local operator system in $\mathcal{A}$ such
that $\mathcal{S}$ generates $\mathcal{A}$ and $\pi :\mathcal{A\rightarrow }%
B\left( \mathcal{H}\right) $ be an irreducible local representation. Then
the following statements are equivalent:

\begin{enumerate}
\item $\mathcal{S}$ separates $\pi ;$

\item There exist $\lambda _{0}\in \Lambda $ and an irreducible
representation $\pi _{\lambda _{0}}:\mathcal{A}_{\lambda _{0}}\mathcal{%
\rightarrow }B\left( \mathcal{H}\right) $ such that $\pi =\pi _{\lambda
_{0}}\circ \pi _{\lambda _{0}}^{\mathcal{A}}$ and for every $\lambda \in
\Lambda $ with $\lambda _{0}\leq \lambda ,$ $\mathcal{S}_{\lambda }$
separates $\pi _{\lambda _{0}}\circ \pi _{\lambda \lambda _{0}}^{\mathcal{A}%
}.$
\end{enumerate}
\end{lemma}

\begin{proof}
Since $\pi :\mathcal{%
A\rightarrow }B\left( \mathcal{H}\right) $ is an irreducible local
representation, by Proposition \ref{A} and Remark \ref{11}, there exist $%
\lambda _{0}\in \Lambda $ and an irreducible representation of $\mathcal{A}%
_{\lambda _{0}}$, $\pi _{\lambda _{0}}:\mathcal{A}_{\lambda _{0}}\mathcal{%
\rightarrow }B\left( \mathcal{H}\right) $, such that $\pi =\pi _{\lambda
_{0}}\circ \pi _{\lambda _{0}}^{\mathcal{A}},$ and for each $\lambda \in
\Lambda $ with $\lambda _{0}\leq \lambda ,$ $\pi _{\lambda _{0}}\circ \pi
_{\lambda \lambda _{0}}^{\mathcal{A}}$ is an irreducible representation of $%
\mathcal{A}_{\lambda }$. 

$\left( 1\right) \Rightarrow \left( 2\right) $ Let $\lambda \in \Lambda $ with $\lambda _{0}\leq
\lambda ,$ $\rho _{\lambda }:\mathcal{A}_{\lambda }\mathcal{\rightarrow }%
B\left( \mathcal{K}\right) $ be an irreducible representation and $V\in
B\left( \mathcal{H},\mathcal{K}\right) $ be an isometry such that $\left(
\pi _{\lambda _{0}}\circ \pi _{\lambda \lambda _{0}}^{\mathcal{A}}\right)
\left( \pi _{\lambda }^{\mathcal{A}}\left( a\right) \right) =V^{\ast }\rho
_{\lambda }\left( \pi _{\lambda }^{\mathcal{A}}\left( a\right) \right) V$
for all $a\in \mathcal{S}$.$\ $Then $\rho _{\lambda }\circ \pi _{\lambda }^{%
\mathcal{A}}:\mathcal{A\rightarrow }B\left( \mathcal{K}\right) $ is an
irreducible local representation and
\begin{equation*}
\pi \left( a\right) =\left( \pi _{\lambda _{0}}\circ \pi _{\lambda \lambda
_{0}}^{\mathcal{A}}\right) \left( \pi _{\lambda }^{\mathcal{A}}\left(
a\right) \right) =V^{\ast }\rho _{\lambda }\left( \pi _{\lambda }^{\mathcal{A%
}}\left( a\right) \right) V=V^{\ast }\left( \rho _{\lambda }\circ \pi
_{\lambda }^{\mathcal{A}}\right) \left( a\right) V
\end{equation*}%
for all $a\in \mathcal{S}$, whence, since $\mathcal{S}$ separates $\pi $, we
deduce that $\pi $ and $\rho _{\lambda }\circ \pi _{\lambda }^{\mathcal{A}}$
are unitarily equivalent. Then, from the above relation, we deduce that $\pi
_{\lambda _{0}}\circ \pi _{\lambda \lambda _{0}}^{\mathcal{A}}$ and $\rho
_{\lambda }$ are unitarily equivalent. Consequently, $\mathcal{S}_{\lambda }$
separates $\pi _{\lambda _{0}}\circ \pi _{\lambda \lambda _{0}}^{\mathcal{A}%
}.$

$\left( 2\right) \Rightarrow \left( 1\right) $ Let $\rho :\mathcal{%
A\rightarrow }B\left( \mathcal{K}\right) $ be an irreducible local
representation and $V\in B\left( \mathcal{H},\mathcal{K}\right) $ be an
isometry such that $\pi \left( a\right) =V^{\ast }\rho \left( a\right) V$
for all $a\in \mathcal{S}$. There exist $\lambda _{1}\in \Lambda $ with $%
\lambda _{0}\leq \lambda _{1}$ and an irreducible representation $\rho
_{\lambda _{1}}:\mathcal{A}_{\lambda _{1}}\mathcal{\rightarrow }B\left(
\mathcal{K}\right) $ such that $\rho =\rho _{\lambda _{1}}\circ \pi
_{\lambda _{1}}^{\mathcal{A}}$.$\ $ Then
\begin{equation*}
\left( \pi _{\lambda _{0}}\circ \pi _{\lambda _{1}\lambda _{0}}^{\mathcal{A}%
}\right) \left( \pi _{\lambda _{1}}^{\mathcal{A}}\left( a\right) \right)
=\pi \left( a\right) =V^{\ast }\rho \left( a\right) V=V^{\ast }\rho
_{\lambda _{1}}\left( \pi _{\lambda _{1}}^{\mathcal{A}}\left( a\right)
\right) V
\end{equation*}%
for all $a\in \mathcal{S}$, whence, since $\mathcal{S}_{\lambda _{1}}$
separates $\pi _{\lambda _{0}}\circ \pi _{\lambda _{1}\lambda _{0}}^{%
\mathcal{A}},$ we deduce that $\pi _{\lambda _{0}}\circ \pi _{\lambda
_{1}\lambda _{0}}^{\mathcal{A}}$ and $\rho _{\lambda _{1}}$ are unitarily
equivalent. Consequently, $\pi $ and $\rho $ are unitarily equivalent.
Therefore, $\mathcal{S}$ separates $\pi .$
\end{proof}

As in the case of operator systems, we obtain a characterization of the
local boundary representations on Hilbert spaces for a local operator system
in terms of their restrictions on the local operator system. The following
theorem is a local version of \cite[Theorem 2.4.5]{A} .

\begin{theorem}
Let $\mathcal{S}$ be a local operator system in $\mathcal{A}$ such that $%
\mathcal{S}$ generates $\mathcal{A}$. Then an irreducible local
representation $\pi :\mathcal{A\rightarrow }B\left( \mathcal{H}\right) $ is
a local boundary representation for $\mathcal{S}$ if and only if

\begin{enumerate}
\item $\left. \pi \right\vert _{\mathcal{S}}$ is a pure local $\mathcal{CP}$
-map;

\item $\pi $ is a finite local representation for $\mathcal{S}$;

\item $\mathcal{S}$ separates $\pi .$
\end{enumerate}
\end{theorem}

\begin{proof}
It follows from Propositions \ref{2} and \ref{3}, Lemma \ref{4} and \cite[%
Theorem 2.4.5]{A}.
\end{proof}

\begin{remark}
Since the local boundary representations for a local \textit{\ operator
system }$\mathcal{S}$ \textit{in a Fr\'{e}chet locally }$C^{\ast }$-algebra
on \textit{quantized Fr\'{e}chet domains }are local boundary representations
on Hilbert spaces (Proposition \ref{1}), Theorem 5.11 \cite{Ar} is a
particular case of the above theorem.
\end{remark}

\begin{acknowledgement}
The author thanks to the referees for the proposed comments which improved
the paper.
\end{acknowledgement}

\end{document}